\documentclass[12pt,leqno]{article}
\usepackage{preamble}
\title{\textsc{Ideally regular categories}}
\author{
  Sandra Mantovani\textsuperscript{*}
  \\
  {\small\texttt{sandra.mantovani@unimi.it}}
  \and
  Mariano Messora\textsuperscript{*}
  \\
  {\small\texttt{mariano.messora@unimi.it}}
}
\date{
\vspace{-.5em}
\footnotesize{\textsuperscript{*}Department of Mathematics, University of Milan, Via Cesare Saldini 50, 20133 Milan, Italy
}
\\
\vspace{1em}
\normalsize{\today}
\vspace{-1em}
}

\begin{document}
\maketitle%
\hrule
\vspace{.4em}
\noindent\textbf{Abstract.} In this note, we propose a generalisation of G.~Janelidze’s notion of an ideally exact category beyond the Barr exact setting. We define an \emph{ideally regular category} as a regular, Bourn protomodular category with finite coproducts in which the unique morphism $0 \to 1$ is \emph{effective for descent}. As in the ideally exact case, ideally regular categories support a notion of ideal that classifies regular quotients. Moreover, they admit a characterisation in terms of monadicity over a \emph{homological} category (rather than a \semiab{} one, as in the exact setting). Examples include \bprot{} \qvars{} of universal algebra in which $0 \to 1$ is effective for descent (such as the category of torsion-free unital rings), all \bprot{} topological varieties with at least one constant (such as topological rings), and all semi-localisations of ideally exact categories. \blfootnote{\href{http://www.tac.mta.ca/tac/volumes/45/11/45-11abs.html}{Published in \emph{Theory Appl.\ Categ.}\ 45.11 (2026), pp.\ 391--400}}
\vspace{.4em}
\hrule
\vspace{.4em}

\vspace{.4em}
\hrule
\section*{Introduction}
In analogy with the classical theory of ideals of unital rings,  G.~Janelidze introduced \emph{ideally exact categories} in \cite{IDE} as a class of categories equipped with a well-behaved notion of `{ideal}' capable of classifying quotient objects. Formally, ideally exact categories 
 are defined as Barr exact, \bprot{} (\cite{OG-BOURN}) categories with finite coproducts such that the unique morphism $0 \to 1$ is a regular epimorphism. For an ideally exact category $\cat A$, the pullback functor along $p\colon0\to1$ between the corresponding slice categories
\[
p^\ast\colon\cat A\cateq\slice A1\to\slice A0
\]
is monadic, and $\slice A0$ is \semiab{} (in the sense of \cite{SEMI-AB}) -- in fact, ideally exact categories can be equivalently characterised as Barr exact categories with finite coproducts that are monadic over a \semiab{} category. A fundamental feature of such a category $\cat A$ is that for any object $A\in\cat A$, the \emph{ideals of $A$} -- defined as normal subobjects of $p^\ast(A)$ -- are in bijective correspondence with the regular quotients of $A$. A paradigmatic example is indeed the category $\Ring$ of unital rings, which is in fact ideally exact. The corresponding functor $p^\ast\colon\Ring\to\slices\Ring0$ (with $0=\Z)$ coincides, up to equivalence, with the forgetful functor $\Ring\to\Rng$ to the \semiab{} category of non-unital rings, thereby recovering the usual notion of an ideal of a ring. Ideally exact categories  provide a wide generalisation of this classical situation, and they encompass many commonly encountered examples, including all \bprot{} varieties with at least one constant, the dual category of any elementary topos, and all coslices of \semiab{} categories. 

The aim of the present work is to extend this framework beyond the Barr exact setting. A leading motivation is the attempt to capture  examples from the world of \bprot{} \qvars{} and topological varieties (\cite{ROQUE-PM,BORCEUX-TOP}), as these are generally not Barr exact categories. To this end, we propose the notion of an \emph{ideally regular category}, defined as a regular, \bprot{} category with finite coproducts in which the unique morphism $0 \to 1$ is \emph{effective for descent} (see for example \cite{FACETS-I}). As effective descent morphisms coincide with regular epimorphisms in the Barr exact context, one immediately sees that every ideally exact category is, in particular, ideally regular. Nevertheless, we show that this broader notion still satisfies analogues of many of the key properties of ideally exact categories established in \cite{IDE} -- although their validity outside the exact setting requires non-trivial reworking of the original arguments. Notably, we prove that ideally regular categories still support a notion of \emph{ideal} -- defined in the same way as in the ideally exact setting -- such that the ideals of an object are in bijective correspondence with its regular quotients (\zcref{thm:ideals}). Moreover, we find that ideally regular categories turn out to be precisely those regular categories with finite coproducts that are monadic over a \emph{homological} (\cite{BBBOOK}) category with finite coproducts (\zcref{idr-e}), thereby paralleling the characterisation of ideally exact categories in terms of monadicity over a \semiab{} category.
Finally, we show that, as desired, this broader context encompasses many new examples, explored in \zcref{sec:def-ex}. Notable instances include all \bprot{} \qvars{} in which $0 \to 1$ is effective for descent (such as the category of torsion-free unital rings), all \bprot{} topological varieties with at least one constant (such as the category of topological unital rings), as well as all \semiloc{}s of ideally exact categories and all coslices of homological categories having finite coproducts. We also provide an example of a \bprot{} \qvar{} where $0 \to 1$ is a regular epimorphism which is not effective for descent. In this case we show that ideals fail to classify quotients, thus proving that even in a well-behaved setting, the effective descent assumption is not automatic and remains essential to the theory.
\section{Definition and examples}
\label{sec:def-ex}
In this section, we formally state an (intrinsic) definition of \emph{ideally regular category} and we provide a list of illustrative examples. A more detailed justification of this definition will follow from the general properties established in the subsequent sections.

Here and throughout, for any category $\cat A$ we use $0_{\cat A}$ and $1_{\cat A}$ to denote the initial and terminal objects of $\cat A$, respectively (dropping the subscript `$\cat A$' when the ambient category is clear for context), and for any object $X\in\cat A$, we write $!_X$ and  $!^X$ for the unique morphisms $0\to X$ and $X\to1$, respectively.

For the background on \bprot{} and homological categories, we refer the reader to \cite{BBBOOK}. For material on descent theory, see for example \cite{FACETS-I}.
\begin{definition}
\label{def-id-reg}
We say that a category $\cat A$ is \emph{ideally regular} if it is a regular, \bprot{} category with finite coproducts such that the unique morphism $0\to1$ in $\cat A$ is effective for descent.
\end{definition}

\begin{example}
    Every ideally exact category is ideally regular, and an ideally regular category is ideally exact if and only if  it is Barr exact.
\end{example}
\begin{example}
    Every homological category with finite coproducts is ideally regular, and an ideally regular category is homological if and only if it is pointed.
\end{example}
\begin{example}
\label{ex-quasivar}
\Qvars{} of universal algebra are ideally regular precisely when they are \bprot{} and the morphism $0\to1$ is effective for descent. \bprot{} \qvars{} are characterised in the same way as \bprot{} varieties, that is by having, for some natural number $n$, constants $e_1,\dots,e_n$, binary terms $t_1,\dots,t_n$ and an $(n+1)$-ary term $t$ such that $t(x,t_1(x,y),\dots,t_n(x,y))=y$ and $t_k(x,x)=e_k$ for $k=1,\dots,n$ (see \cite{GRAN-ROS,ROQUE-PM}). In particular, \qvars{} whose theory contains a group operation are \bprot{}.

Whether the map $0 \to 1$ is effective for descent in a \qvar{} is more subtle. There exist various criteria in the literature for detecting when a morphism is effective for descent (see again \cite{FACETS-I} and the references therein). In the case of a \qvar{} $\varty Q$, it is well-known that there exists a variety $\varty V$ such that $\varty Q$ is a full subcategory of $\varty V$ closed under finite limits. If $p\colon0_{\varty Q} \to 1_{\varty Q}$ is a regular epimorphism in $\varty Q$, then $p$ is also a regular epimorphism in $\varty V$, since in both categories regular epimorphisms are simply surjective maps (\cite{ROQUE-ED}); hence $p$ is effective for descent in $\varty V$. By applying \cite[Corollary in 2.7]{FACETS-I} together with the fact that $1_{\varty Q}$ is terminal in both $\varty Q$ and $\varty V$, one finds that $p$ is effective for descent in $\varty Q$ if and only if for every $A \in \varty V$ such that $0_{\varty Q} \times A \in \varty Q$ (with `$\times$' denoting the product in $\varty V$), one has $A \in \varty Q$. Using this criterion, we obtain the following examples.
\begin{enumerate}[(a)]
    \item \label{it-torsion-free-rings} Let $\varty Q$ be  the (\bprot{}) \qvar{} of unital torsion-free rings, that is unital rings satisfying $\impl {n\cdot x=0}{x=0}$ for all positive integers $n$. Here, $0=\Z$ and the map $0\to 1$ is clearly surjective. Take $\varty V$ to be the variety of all unital rings, so that $\varty Q$ is a full subcategory of $\varty V$ closed under limits. Suppose $R\in\varty V$ is such that $\Z\times R$ is torsion-free. If $x\in R$ is such that $n\cdot x=0$ for some positive integer $n$, then $n\cdot (0,x)=0$ in $\Z\times R$, hence $(0,x)=0$, and so $x=0$. Thus $R$ is torsion-free and $0\to 1$ is effective for descent. We conclude that the \qvar{} of torsion-free unital rings is ideally regular. Note, however, that $\varty Q$ is not a Barr exact category.
    \item In a similar fashion, one checks that the \qvar{} of reduced unital rings (i.e.\ unital rings satisfying $\impl{x^2=0}{x=0}$) and the \qvar{} of unital torsion-free algebras over a fixed commutative unital ring $R$ (i.e.\ unital $R$-algebras satisfying $\impl{r\cdot x=0}{x=0}$ for all non-zero-divisors $r \in R$) are also ideally regular.
    \item\label{it:w-Q} Consider instead the (\bprot{}) \qvar{} $\varty Q$ of rings of characteristic 0 (defined by the implications $\impl{n\cdot 1=0}{1=0}$ for all positive integers $n$). The initial object of $\varty Q$ is $\Z$ with the usual operations, and once again the morphism $0_{\varty Q}\to 1_{\varty Q}$ is surjective. Note that $\varty Q$ is a full subcategory closed under limits of the variety $\Ring$ of unital rings, and observe that, for all rings $A\in\Ring$, the product ring $\Z\times A$ is always a ring of characteristic 0. We conclude that the regular epimorphism $0_{\varty Q}\to 1_{\varty Q}$ is \emph{not} effective for descent, showing that $\varty Q$ is not ideally regular.   
\end{enumerate}
\end{example}
\begin{example}
    A \semiloc{} (as in \cite{MANTOVANI-SL}) of an ideally exact category is ideally regular. Indeed, if $\cat A$ is a \semiloc{} of a Barr exact, \bprot{} category $\cat B$, by \cite{GRAN-LACK} $\cat A$ is regular and \bprot{}, and effective descent morphisms in $\cat A$ coincide with regular epimorphisms. If moreover $\cat B$ has finite coproducts, then so does $\cat A$, being reflective in $\cat B$. Writing $F\colon \cat{B}\to \cat{A}$ for the left adjoint to the inclusion $\cat A\hookrightarrow\cat B$, we have $F(0_{\cat{B}}\to 1_{\cat{B}})=(0_{\cat{A}}\to 1_{\cat{A}})$ since $\cat{A}$ is closed under limits in $\cat{B}$ and $F$ preserves colimits. Furthermore, if $0_{\cat{B}}\to 1_{\cat{B}}$ is a regular epimorphism, then so is $0_{\cat{A}}\to 1_{\cat{A}}$, since $F$ preserves regular epimorphisms; hence it is also effective for descent. As an example, the category of  unital torsion-free rings (in the sense of \zcref{ex-quasivar}.\ref{it-torsion-free-rings}) is a (non-exact) \semiloc{} of the category of unital rings.  This follows from the fact that the torsion elements of any ring form an ideal. 

    Note however that not all ideally regular categories are \semiloc{}s of ideally exact categories. For instance, consider the (\bprot{}) \qvar{} $\varty Q$ of abelian groups or unital rings satisfying $\impl{4\cdot x=0}{2\cdot x=0}$ (example borrowed from \cite{FACETS-I,ROQUE-ED}). The morphism $0_{\varty Q}\to1_{\varty Q}$ is effective for descent (this is obvious in the case of abelian groups; for rings, use the criterion in \zcref{ex-quasivar}), and hence $\varty Q$ is ideally regular. However, as shown in \cite{FACETS-I,ROQUE-ED}, the canonical morphism $\Z\to\Zm2$ is a regular epimorphism which is not effective for descent, and hence $\varty Q$ cannot be a \semiloc{} of a Barr exact category.
\end{example}
\begin{example}
    Let $\theory T$ be an \emph{ideally exact algebraic theory}, i.e.\ a \bprot{} algebraic theory with at least one constant. Then the category $\cat A=\AlgThr T\Top$ of models of $\theory T$ in the category of topological spaces is ideally regular. Indeed, we know from \cite{BORCEUX-TOP} that $\cat A$ is regular and \bprot{}. Effective descent morphisms in $\cat A$ coincide with regular epimorphisms and consist of open surjective maps (see for instance \cite{JOHNSTONE-PEDICCHIO,GRAN-ROSSI}). Since $0_{\cat A}$ is non-empty (as $\theory T$ has at least one constant) and $1_{\cat A}$ is always a singleton, it follows that the morphism $0_{\cat A}\to 1_{\cat A}$ is always open and surjective.
    Note moreover that $\cat A$ is in general not Barr exact, since an effective equivalence relation on an object $A$ needs to be equipped with the topology induced by the product topology on $A\times A$. This is not the case for a general internal equivalence relation in $\AlgThr T\Top$.
\end{example}

\section{Equivalent characterisations}
In this section we aim to characterise ideally regular categories in terms of properties of a monadic functor over a homological category, in analogy with results from \cite{IDE}.

Before doing so, we fix some notation concerning slice categories and recall some related facts from \cite{IDE} which will be needed in what follows. 

Given a category $\cat A$ and an object $X\in\cat A$, objects of the comma category $\slice AX$ will be written as pairs $(A,\alpha)$, with $\alpha\colon A\to X$, and morphisms in $\slice AX$ will be written as morphisms of the underlying objects. When $\cat A$ has pullbacks, for any morphism $p\colon E\to B$ in $\cat A$, we will denote by
\begin{equation*}
    p_!\adj p^\ast\colon\slice AB\to\slice AE
\end{equation*}
the \emph{adjunction associated to $p$}, given, for any $(D,\delta)\in \slice AE$ and $(A,\alpha)\in\slice AB$, by
\[
\hskip\textwidth minus \textwidth
p_!(D,\delta)=(D,\comp \delta p),
\hskip\textwidth minus \textwidth
p^\ast (A,\alpha)=(P,\pi_1),
\hskip\textwidth minus \textwidth
\]
with $\pi_1\colon P\to A$ denoting the projection on $A$ of the pullback of $\alpha$ along $p$. Let us now recall the following theorem from \cite{IDE}.

\begin{theorem}[{\cite[Theorem~2.4]{IDE}}]
\label{canonical-monad}
Let $\monad T=(T,\eta, \mu)$ denote the monad associated to the above adjunction. The following properties hold.
\begin{itemize}
    \item $\monad T$ is \emph{cartesian}, meaning that $T$ preserves pullbacks, and all the naturality squares associated to $\eta$ and $\mu$ are pullbacks;
    \item when $\cat A$ is regular, $T$ preserves regular epimorphisms and kernel pairs;
    \item when $\cat A$ is \bprot{} and has finite coproducts, $\monad T$ is \emph{essentially nullary}, meaning that for all objects $X\in\cat A$, the morphism 
    \[(T(!_X), \eta_X)\colon T(0)+X\to T(X)\]
    is a strong epimorphism.
\end{itemize}    
\end{theorem}
We are now ready to state the main result of this section.
\begin{theorem}
\label{idr-e}
Let $\cat A$ be any category. The following are equivalent.
\begin{enumerate}[(a)]
    \item \label{idr-e-a}$\cat A$ is ideally regular;
    \item \label{idr-e-b}$\cat A$ is regular, has finite coproducts, and there exists a monadic functor $\cat A\to \cat X$ where $\cat X$ is homological and admits finite coproducts;
    \item \label{idr-e-c}there exists a monadic functor $\cat A\to\cat X$ such that both $\cat A$ and $\cat X$ have finite coproducts, $\cat X$ is homological, and the underlying functor of the corresponding monad preserves regular epimorphisms and kernel pairs.  
\end{enumerate}
\end{theorem}
\begin{proof}
    Let $p$ denote the unique morphism $0\to 1$ in $\cat A$, and let $p_!\adj p^\ast$ be the associated adjunction.
    
    When \ref{idr-e-a} holds, since $p$ is effective for descent, the functor  
    \[p^\ast\colon\cat A\cateq\slice A1\to\slice A0\] 
    is monadic, with $\slice A0$ obviously pointed. Moreover, since $\cat A$ is regular and \bprot{}, so is its slice $\slice A0$, which is hence homological. Lastly, once again using the regularity of $\cat A$, it follows from \zcref{canonical-monad} that $\fmcomp{p_!,p^\ast}$ preserves regular epimorphisms and kernel pairs. This proves that \ref{idr-e-a}$\implies$\ref{idr-e-c}.
    
    Next, if \ref{idr-e-c} holds, then $\cat A$ is regular since it is equivalent to the category of algebras over a monad whose underlying functor preserves regular epimorphisms (see for example \cite{VITALE94}). Hence, \ref{idr-e-b} follows.
    
    Finally, if \ref{idr-e-b} holds, assume that $U\colon\cat A\to \cat X$ is a monadic functor where $\cat A$ is regular, $\cat X$ is homological, and both categories admit finite coproducts. We immediately deduce that $\cat A$ is  \bprot{}. Let $F$ be a left adjoint to $U$, and for simplicity write $F(0)=0$. We obtain the following diagram of adjunctions,
    \begin{cdiag*}
       \cat A\arrow[r, "p^\ast", yshift=.2em] \arrow[d, "U"', xshift=-.2em]&\slice A 0 \arrow[l, "p_!", yshift=-.2em]\arrow[dl, bend left, "U^0", xshift=.17em, yshift=-.17em, shorten=.3]
       \\ 
       \cat X\arrow[u, "F"', xshift=.2em]\arrow[ur,  bend right,yshift=0.17em, xshift=-.17em, shorten=.3em, "F^0"]
    \end{cdiag*}
    where $F^0\adj U^0$ are given by:
    \[
    \hskip\textwidth minus \textwidth
    F^0(X)=(F(X),F(!^X)),
    \hskip\textwidth minus \textwidth
    U^0(A,\alpha)=\ker(U(\alpha))
    \hskip\textwidth minus \textwidth
    \]
    (see \cite{JANELIDZE-91,IDE}). In the proof of \cite[Theorem~2.6]{IDE} it is shown that $U^0$  reflects isomorphisms, only relying on the \bproty{} of $\cat X$. Since $U\iso\fmcomp{p^\ast,U^0}$ is monadic and $U^0$ is conservative, it follows from \cite[Propositions~4 and 5]{BOURN91} that $p^\ast$ is also monadic, and hence $p$ is effective for descent.
\end{proof}

With \zcref{idr-e} in place, we can make the following observations. In particular, \zcref{special-monad} is based on \cite[Theorem~3.3.b]{IDE}, \zcref{rmk:cartesian} on \cite[Remark~3.8]{IDE} and \zcref{ex:coslice-idr} on \cite[Example~3.5]{IDE}.
\begin{corollary}
\label{special-monad}
    Let $\cat A$ be an ideally regular category. Then there exists a monadic functor $\cat A\to \cat X$ such that $\cat X$ is homological and has finite coproducts, the corresponding monad is cartesian, essentially nullary and such that its underlying functor preserves regular epimorphisms and and kernel pairs.
\end{corollary}
\begin{proof}
It suffices to use the same monadic functor $p^\ast$ as in the proof of \zcref{idr-e} and recall \zcref{canonical-monad}.
\end{proof}
\begin{remark}
    Note that, in \zcref{idr-e}, unlike in the analogous Theorem~3.1 in \cite{IDE}, the existence of finite coproducts must be assumed explicitly even in \zcref{idr-e-c}. Moreover, as far as we are aware, ideally regular categories do not generally admit finite colimits -- unlike ideally exact categories -- even when they are monadic over a homological category that does (cf.\ \cite[Theorem~3.3.a]{IDE}).
\end{remark}
\begin{remark}
\label{rmk:cartesian} For an ideally regular category $\cat A$, the monadic adjunction $\cat A\rightleftarrows\cat X$ of \zcref{idr-e}.\ref{idr-e-b} and \ref{idr-e-c} coincides, up to equivalence, with the adjunction associated to $0_{\cat A}\to1_{\cat A}$ if and only if the unit of the adjunction is cartesian. This follows from \cite[Theorems~2.4 and 2.6]{IDE}.
\end{remark}
\begin{example}
\label{ex:coslice-idr}
    \begin{enumerate}[wide, labelindent=0pt,nosep,label=(\alph*)]
        \item If $\cat X$ is homological with finite coproducts and $X\in\cat X$, then the coslice $\coslice XX$ is ideally regular. This follows from the fact that the codomain functor $\coslice XX\to\cat X$ is monadic. Furthermore, from \zcref{special-monad}, we know that every ideally regular category is equivalent to a category of algebras $\AlgMon XT$ over an essentially nullary monad $\monad T$ on a homological category $\cat X$ -- and \cite[Theorem~1.1]{CARBONI-JANELIDZE-95} tells us that $\monad T$ is nullary if and only if $\AlgMon XT$ is canonically equivalent to some $\coslice XX$.
        \item A slice $\slice XX$ of a homological category with finite coproducts is ideally regular if and only if $X=0$.
    \end{enumerate}
\end{example}
\begin{remark}
    For any category $\cat A$, the following are equivalent:
\begin{enumerate*}[(a)]
    \item $\cat A$ is finitely complete and \bprot{}, it has an initial object, and the morphism $0\to1$ in $\cat A$ is effective for descent;
    \item there exists a monadic functor $\cat A\to\cat X$ with $\cat X$ pointed, finitely complete and \bprot{}.
\end{enumerate*}   This follows by arguing as in the proof of \zcref{idr-e}.
\end{remark}
\section{Ideals and quotients}
In this section, we show that ideally regular categories support a notion of ideal that classifies regular quotients, in direct analogy with the ideally exact case. To this end, we will make use of the following property of essentially nullary monads, recalled from \cite{IDE}.
\begin{proposition}[{\cite[Corollary~4.2]{IDE}}]
\label{prop-rel-ess-null}
    Let $\monad T$ be an essentially nullary monad on a category with finite coproducts, and let $(X,\xi)$ be a $\monad T$-algebra. Every reflexive relation
    \begin{cdiag*}
        R\arrow[r, "r_0", yshift=.4em]\arrow[r, "r_1"', yshift=-.4em]&X\arrow[l, "e" description]
    \end{cdiag*}
    on $X$ admits a (unique) $\monad T$-algebra structure making $r_0$, $r_1$ and $e$ morphisms of $\monad T$-algebras.
\end{proposition}
We can now establish the following result.
\begin{theorem}
\label{thm:ideals}
Let $\cat A$ be an ideally regular category. 
Fix a monadic functor $U\colon\cat A\to\cat X$ such that 
\begin{itemize}
    \item $\cat X$ is homological and has finite coproducts;
    \item the monad $\monad T$ corresponding to $U$ is essentially nullary;
    \item the functor $T$ underlying $\monad T$ preserves regular epimorphisms and kernel pairs
\end{itemize}
(the existence of such a functor  is guaranteed by \zcref{special-monad}).
For every object $A$ in $\cat A$, the following (possibly large) posets are canonically isomorphic to each other:
\begin{enumerate}[(a)]
    \item \label{Quot-A}$\Quot {\cat A}A$, consisting of quotient objects of $A$;
    \item \label{Rel-A}$\EERel {\cat A}A$, consisting of effective equivalence relations on $A$;
    \item \label{Rel-UA}$\EERel {\cat X}{U(A)}$, consisting of effective equivalence relations on $U(A)$;
    \item \label{Quot-UA}$\Quot {\cat X}{U(A)}$, consisting of quotient objects of $U(A)$;
    \item \label{Sub-UA}$\NSub {\cat X}{U(A)}$, consisting of normal subobjects of $U(A)$.
\end{enumerate}
\end{theorem}
\begin{proof}
    The isomorphism between the posets in \zcref[typeset=ref]{Quot-A,Rel-A} is known, as  are those between the posets in \zcref[typeset=ref]{Quot-UA,Rel-UA,Sub-UA}. To prove that the posets in \zcref[typeset=ref]{Rel-A,Rel-UA} are isomorphic, let us identify $\cat A$ with the Eilenberg-Moore category $\AlgMon X T$ of algebras over $\monad T$, so that $U\colon \AlgMon XT\to \cat X$ is the canonical forgetful functor. Under this identification, any object $A\in\cat A$ corresponds to some algebra $(X,\xi)\in\AlgMon XT$, with $\xi\colon T(X)\to X$ in $\cat X$. Since $U$ is a right adjoint, it induces an injective morphism
    \begin{cdiag*}
        \EERel{\AlgMon XT}{(X,\xi)}\rar&\EERel {\cat X} X.
    \end{cdiag*}
    Now, in $\cat X$, consider an effective equivalence relation on $X$ together with its coequaliser:
    \begin{cdiag*}
        R\arrow[r, "r_0", yshift=.2em]\arrow[r, "r_1"', yshift=-.2em] &X\rar{q}&Q.
    \end{cdiag*}
    By \zcref{prop-rel-ess-null}, $R$ admits a $\monad T$-algebra structure $\rho\colon TR\to R$ making $r_0$ and $r_1$ morphisms of $\monad T$-algebras. We aim to show that $((R,\rho),r_0,r_1)$ is an effective relation in $\AlgMon XT$. Consider the solid arrows in the following diagram.
     \begin{cdiag*}
        R\arrow[r, "r_0", yshift=.2em]\arrow[r, "r_1"', yshift=-.2em] &X\rar{q}&Q
        \\
        T(R)\arrow[r, "T(r_0)", yshift=.2em]\arrow[r, "T(r_1)"', yshift=-.2em]\arrow[u, "\rho"]&T(X)\arrow[u, "\xi"description]\arrow[r, "T(q)"'] & T(Q)\arrow[u, dashed, "\chi"']
    \end{cdiag*}
    Since $T$ preserves kernel pairs and regular epimorphisms, $T(q)$ is the coequaliser of $T(r_0)$ and $T(r_1)$. Now,  $\mcomp{T(r_0),\xi,q}=\mcomp{T(r_1),\xi,q}$ and hence there exists a unique $\chi\colon T(Q)\to Q$ such that $\comp \xi q=\comp{T(q)}\chi$. Using that $q$ and $T(T(q))$ are epimorphisms, one easily checks that $(Q,\chi)$ is a $\monad T$-algebra, and by construction, $q$ is a morphism of $\monad T$-algebras from $(X,\xi)$ to $(Q,\chi)$. Since monadic functors reflect and create limits, we conclude that $((R,\rho),r_0,r_1)$ is the kernel pair of $q\colon (X,\xi)\to(Q,\chi)$ in $\AlgMon XT$, concluding the proof.
\end{proof}
\begin{remark}
    In a homological category, the poset of normal subobjects of a given object forms a \mtsl{}, with the meet of two normal subobjects given by their pullback. Consequently, all the posets of \zcref{thm:ideals} are in fact \mtsl{}s. Moreover, the poset of normal subobjects of an object becomes a lattice as soon as the ambient category also admits pushouts of regular epimorphisms along regular epimorphisms. This condition is automatically satisfied when the category is Barr~exact (see \cite{BBBOOK}).
\end{remark}
In light of \zcref{thm:ideals}, the definition of an \emph{ideal} from \cite{IDE} extends naturally to the context of ideally regular categories. We point out, however, that in \cite[Theorem~4.3]{IDE}, the functor underlying the monad associated to $U$ is not required to preserve regular epimorphisms and kernel pairs, unlike in our \zcref{thm:ideals}.
\begin{definition}
    In the situation of \zcref{thm:ideals}, given an object $A$ in $\cat A$, a normal monomorphism in $\cat X$ with codomain $U(A)$ will be called a \emph{$U$-ideal} of $A$. In particular, if $U$ is the pullback functor $\cat A\to\slice A0$ along $0\to 1$, then $U$-ideals of $A$ will be simply called \emph{ideals} of $A$.
\end{definition}
\begin{example} When $\cat A$ is, for instance, the category of unital torsion-free rings (as in \zcref{ex-quasivar}.\ref{it-torsion-free-rings}) or the category of unital topological rings, then the corresponding functor $p^\ast\colon\cat A\to\slice A0$ coincides, up to equivalence, with the obvious forgetful functor to the category of non-unital torsion-free rings and non-unital topological rings, respectively.  In these cases, we therefore recover the natural notion of ideal one would expect.
\end{example}
\begin{remark}
    In any category $\cat A$ with finite limits and an initial object, one can still define the  \emph{ideals} of an object $A\in\cat A$ as the normal subobjects of $p^\ast(A)$, with $p\colon 0\to1$. However, such ideals will generally fail to classify regular quotients, even under considerably strong assumptions on $\cat A$. Consider for instance the \qvar{} $\varty Q$ of \zcref{ex-quasivar}.\ref{it:w-Q} of rings of characteristic 0. We have seen that $\varty Q$ is a regular, \bprot{} category with finite coproducts such that $p\colon 0=\Z\to1$ is a regular epimorphism but not an effective descent morphism. One can show that, up to equivalence, the corresponding functor $p^\ast$ coincides with the obvious forgetful functor to the category of non-unital rings. Yet the ring $\Z$ has no non-trivial quotients of characteristic 0, while still admitting infinitely many ideals.
\end{remark}
\section*{Acknowledgments}
This research was conducted while both authors were affiliated with INdAM -- Istituto Nazionale di Alta Matematica ‘Francesco Severi’, Gruppo Nazionale per le Strutture Algebriche, Geometriche e le loro Applicazioni (GNSAGA). 
\printbibliography
\end{document}